\documentclass[11pt]{amsart}
\usepackage{amsmath,amsthm}
\usepackage{amsfonts}
\usepackage{verbatim}
\usepackage{amssymb}
\usepackage{txfonts}
\usepackage{amscd}
\usepackage{bbm}
\usepackage{hyperref}
\usepackage{mathrsfs}
\usepackage{graphicx}
\usepackage{diagrams}
\usepackage{color}

\newcommand{\set}[1]{\,\left\{#1\right\}}
\newcommand{\setd}[2]{\,\left\{#1\ \colon\ #2\right\}}

\newtheorem{theorem}{Theorem}[section]
\newtheorem{corollary}[theorem]{Corollary}
\newtheorem{lemma}[theorem]{Lemma}
\newtheorem{proposition}[theorem]{Proposition}
\newtheorem{definition}[theorem]{Definition}

\newtheorem{remark}[theorem]{Remark}

\newcommand{\Vol}{\operatorname{Vol}}

\newcommand{\isod}{D}
\newcommand{\lin}{\mathrm{lin}}
\newcommand{\rad}{\mathrm{rad}\,}
\newcommand{\finsup}{\mathbb{F}}
\newcommand{\uf}{\operatorname{uf}}

\def\G{\mathscr{G}}
\def\F{\mathscr{F}}

\newcommand{\RR}{\mathbb{R}}
\newcommand{\ZZ}{\mathbb{Z}}

\newcommand{\NN}{\mathbb{N}}

\newarrow{Dotsto}{}{dash}{}{dash}>

\title{Controlled coarse homology and isoperimetric inequalities}

\author{Piotr W. Nowak}
\address{Department of Mathematics, Texas A\&M University, College Station, TX 77843, USA}
\email{pnowak@math.tamu.edu}

\author{J\'an \v{S}pakula}

\address{Fachbereich Mathematik, Universit\"at M\"{u}nster,
Einsteinstr.\ 62, 48149 M\"{u}nster, Germany}
\email{jan.spakula@uni-muenster.de}

\subjclass[2000]{Primary; Secondary;}


\date{May 2009}

\begin{document}
\begin{abstract}
We study a coarse homology theory with prescribed growth conditions. 
For a finitely generated group $G$ with the word length metric this homology theory 
turns out to be related to amenability of $G$. We characterize vanishing of a certain
fundamental 
class in our homology in terms of
an isoperimetric inequality on $G$ and show that on any group at most linear
control is needed for this class to vanish. The latter is a 
homological version of the classical Burnside problem for infinite groups, with
a positive solution. As applications
we characterize existence of primitives of the volume form with prescribed growth
and show that coarse homology classes obstruct weighted Poincar\'{e} inequalities.

\end{abstract}

\maketitle

\section{Introduction}
Isoperimetric inequalities are a fundamental tool in analysis and differential
geometry. Such inequalities, including Sobolev and
Poincar\'{e} inequalities and their numerous generalizations,
 have a large number of applications in various settings. In this paper we 
are interested in a discrete isoperimetric inequality studied 
by \.{Z}uk \cite{zuk} and later by Erschler \cite{erschler}. It is of the form
$$\#A \le C \sum_{x\in \partial A} f(d(x,x_0))$$
for a fixed, non-decreasing real function $f$, a fixed point $x_0$ and a constant $C>0$. 
The purpose of our work is to explore the connection between the inequality and
one of the  fundamental large-scale invariants, coarse homology.

Coarse homology and cohomology were first introduced by Roe in
\cite{roe-coarse.coho} for the purposes of index theory and allowed to formulate
the coarse index on open manifold using assembly maps from coarse $K$-homology to the $K$-theory of appropriate $C^*$-algebras (see also \cite{block-weinberger-survey,roe-cbms,roe-lectures}). 
This approach proved
to be very successful in attacking  various
problems in geometry and topology of manifold such as the Novikov conjecture, positive scalar curvature problem
or the zero-in-the-spectrum conjecture.
Block and Weinberger  \cite{block-weinberger}
introduced and studied a uniformly finite homology theory, where they
considered only those chains in Roe's coarse homology whose coefficients  are bounded. 
This homology theory turned out to have
many applications since vanishing of the $0$-dimensional uniformly finite homology group
characterizes amenability. Using this fact Block and Weinberger related it to 
the existence of aperiodic tilings and positive scalar curvature metrics on complete Riemannian
manifolds \cite{block-weinberger}. Later Whyte used it to show the existence of bijective 
quasi-isometries (i.e. bilipschitz equivalences) between non-amenable groups \cite{whyte}. Other applications
can be found in e.g.\ \cite{attie-block-weinberger,attie-hurder,whyte-ahat}.

The homology theory we study here is a controlled homology theory for spaces
with bounded geometry, where the upper bound on the chains' growth type is specified 
and represented by a fixed, non-decreasing function $f$.  The case when $f$ is constant gives
the uniformly finite homology of Block and Weinberger mentioned above, but  our main
interest is in the case when the chains are unbounded (see Section \ref{section : definitions} for a precise definition). Roe's coarse homology is also defined using unbounded chains, but 
without any control on the coefficients' growth (it is a coarsened version of the locally finite homology). 
Our homology is a quasi-isometry invariant and contains information about the large-scale 
structure of a metric space.

For the most part, we
restrict our attention to the case of finitely generated groups
and we are mainly interested in vanishing of the fundamental class $[\Gamma]=\sum_{x\in
 \Gamma}[x]$ in the $0$-th homology group $H_0^f(\Gamma)$.  The main 
 theme of our work is that the vanishing of this class describes ``how amenable''
 a group is, through the isoperimetric inequality above. In the light of  this philosophy and 
 \cite{zuk} one can expect 
 that on any group killing the fundamental
 class should require a 1-chain with at most linear growth. Our first result  
 confirms this with a direct construction of a linearly growing $1$-chain whose boundary
 is the fundamental class.
This fact restricts the class of growth types that are potentially
interesting to those with growth between constant and linear. As we explain 
at the end of Section \ref{section : an explicit 1-chain}, the theorem can be viewed as a weaker, homological version of the classical
Burnside problem for infinite groups with a positive solution. 
To the best of our knowledge this is the strongest 
result in this direction which is true for all infinite, finitely generated groups.

The result in  section 
\ref{section : homology and inequalities} states that the isoperimetric inequality holds on 
a given space if and only if the fundamental class vanishes in the
corresponding coarse homology group and it is the second main result of the paper.
This characterization generalizes the result of \cite{block-weinberger} which 
arises as the case of bounded control and also gives a homological perspective on 
 \.{Z}uk's result \cite{zuk}.
Moreover, using the results of Erschler \cite{erschler}
we obtain explicit examples of groups for which the fundamental
class bounds 1-chains growing much slower than linearly.

In section \ref{section : examples}
we discuss explicit examples of amenable groups for which 
we can compute explicitly the control function $f$, for which
the fundamental class vanishes in $H_0^f(\Gamma)$. 
The estimates are obtained using invariants such 
as isodiametric and isoperimetric profiles.
These invariants are well-known and widely studied (see
e.g.\ \cite{erschler,nowak-exactiso,pittet-saloff-coste-downunder,
pittet-saloff-coste-survey,saloff-coste-notices,saloff-coste-survey}) and each of them
can be used to estimate the growth
of a 1-chain which bounds the fundamental class. 

Our theory has some interesting applications to the problem of finding a primitive  of a differential 
form on a universal cover of a compact manifold with 
prescribed growth. The question of finding such primitives was studied 
by Sullivan \cite{sullivan}, Gromov \cite{gromov-proceedings} and Brooks \cite{brooks} in a setting which  relates to that of \cite{block-weinberger}, and later
in \cite{sikorav,zuk}.
Our results give exact estimates on this growth in the case when the group in question
is amenable.
Namely, we can use all of the above invariants
as large-scale obstructions to finding such primitives. For instance, if $\Gamma$ is an amenable group which 
has finite asymptotic dimension of linear type then the volume form on the universal cover
of a compact manifold $M$ with the fundamental group $\Gamma$ cannot have a primitive 
of growth slower than linear. Our methods in particular give a different proof
of a theorem of Sikorav characterizing the growth of primitives of differential forms \cite{sikorav}.

Another connection is to weighted Poincar\'{e} inequalities
studied by Li and Wang \cite{li-wang}. These inequalities are intended as a weakening
of a lower bound on the positive spectrum of the Laplacian and were used in \cite{li-wang}
to prove various rigidity results for open manifolds. It turns out that our isoperimetric inequalities and
weighted Poincar\'{e} inequalities are closely related. We use this fact to show that the 
vanishing of the fundamental class is an obstructions to Poincar\'{e} inequalities with certain 
weights, 
which do not decay fast enough.

In a future paper we will use the theory presented here together with surgery theory to study 
positive scalar curvature on open manifolds, Pontrjagin classes and distortion of diffeomorphisms.\\

We are grateful to John Roe, Shmuel Weinberger and Guoliang Yu for valuable discussions and suggestions. 

\setcounter{tocdepth}{1}
\tableofcontents

\section{Coarse homology with growth conditions}
\label{section : definitions}
We will be considering metric spaces which are (uniformly) discrete in the sense that
 there exists a constant $C>0$ such that $d(x,y)\ge C$
for any distinct points $x,y$ in the given space. Additionally we assume that all our spaces have bounded geometry and are quasi-geodesic. A discrete metric space $X$ has bounded
geometry if for every $r>0$ there exists a constant $N_r>0$ such that
$\# B(x,r)\le N_r$ for every $x\in X$.  $X$ is said to be quasi-geodesic 
if for every $x,y\in X$ there exists a sequence of points $\set{x=x_0,x_1,\dots,x_{n-1}, x_n=y}$ such that $d(x_i,x_{i+1})\le 1$ and $n\le d(x,y)$ (this a slightly stronger definition than in the literature, but essentially equivalent).
Examples of such spaces are finitely generated groups with word length metrics.

\subsection{Chain complex, homology and large-scale invariance}
Let us first introduce some notation. Let $X$ be a discrete metric space with bounded geometry 
and let $e\in X$ be
fixed.  For $\overline{x}=(x_1,\dots,x_n), \overline{y}=(y_1,\dots,y_n)\in X^n$ we define the distance
$$d(\overline{x},\overline{y})=\max_i d(x_i,y_i),$$
and the length of $\overline{x}$ as $\vert \overline{x}\vert=d(\overline{x},\overline{e})$ where
$\overline{e}=(e,\dots, e)$. Let $f:[0,\infty)\to[0,\infty)$ be a non-decreasing
function. For technical reasons we will assume that $f(0)=1$ and that for every $K> 0$ there exists $L>0$ such that
\begin{equation}\tag{$f_1$}
\label{equation : f has the same growth under shifts}
f(t+K)\le Lf(t)
\end{equation}
for every $t>0$. 
We will also assume that for every $K>0$
there exists a constant $L>0$ such that
\begin{equation}\tag{$f_2$}\label{equation : f(Ct) le Df(t)}
f(K t)\le Lf(t)
\end{equation}
for every $t>0$. Both conditions are mild and are satisfied by common sub-linear growth types such as powers
and logarithms.

We shall define homology with coefficients in $\RR$, but the definition of
course makes sense for any normed ring, in particular $\ZZ$.
We represent the chains in two ways: As a formal sum $c=\sum_{\overline{x}\in
X^{n+1}}c_{\,\overline{x}}\overline{x}$, $c_{\,\overline{x}}\in \RR$; or as a
function $\psi:X^{n+1}\to \RR$. In both cases, we think of $\overline{x}\in
X^{n+1}$ as of simplices. In particular, we require the property
$\psi(\overline{x})=(-1)^{N(\sigma)}\psi(\sigma(\overline{x}))$, where $\sigma$
is a permutation of
the simplex $\overline{x}\in X^{n+1}$ and $N(\sigma)$ is the number of
transpositions needed 
to obtain $\sigma(\overline{x})$ from $\overline{x}$. For $1$-chains this simply
reduces to 
$\psi(x,y)=-\psi(y,x)$.

Given a chain
$c=\sum_{\overline{x}\in X^{n+1}}c_{\,\overline{x}}\overline{x}$ the
\emph{propagation} $\mathscr{P}(c)$ is the
smallest number $R$ such that $c_{\,\overline{x}}=0$ whenever $d(\overline{x},\Delta_{n+1})\ge R$ for
$\overline{x}\in X^{n+1}$, where $\Delta_{n+1}$ denotes the diagonal in $X^{n+1}$. 

We now define the chain complex. Denote
$$C^{f}_n(X)=\setd{c=\sum_{\overline{x}\in X^{n+1}}c_{\,\overline{x}}\overline{x}}
{\mathscr{P}(c)<\infty\text{ and } \vert c_{\,\overline{x}}\vert \le K_c
  f(\vert \overline{x}\vert)},$$ where $K_c$ is a constant which depends on
$c$ and the coefficients $c_{\,\overline{x}}$ are real numbers. One can easily check using 
(\ref{equation : f has the same growth under shifts}) that
$C_n^f(X)$ is a linear space which does not depend on the choice of the base
point.

We define a differential
$\partial:C^f_n(X)\to C^f_{n-1}(X)$ in a standard way on simplices as
$$\partial([x_0,x_1,\dots,x_n])=\sum_{i=0}^n (-1)^i [x_0,\dots,\hat{x_i},\dots, x_n]$$
and by extending linearly. It is easy to check that 
(\ref{equation : f has the same growth under shifts})
guarantees that the differential is well-defined.
Thus we have a chain complex $\set{C_i^f(X),\partial}$ 
and we denote its homology by $H_*^f(X)$. In particular, if $f\equiv \mathrm{const}$, we obtain uniformly finite homology of Block and Weinberger \cite{block-weinberger} (with real coefficients),
which we denote by $H_*^{\uf}(X)$ and if $f$ is linear we will write $H^{\lin}_*(X)$.

We now turn to functorial properties of  $H_*^f(X)$. First, we recall standard definitions.
\begin{definition}
Let $X$ and $Y$ be metric spaces. A map $F:X\to Y$ is a \emph{coarse equivalence}
if there exist non-decreasing functions $\rho_-,\rho_+:[0,\infty)\to[0,\infty)$ such that
$$\rho_-(d_X(x,y))\le d_Y(F(x),F(y))\le \rho_+(d_X(x,y))$$
for all $x,y\in X$ and there exists a $C>0$ such that for every $y\in Y$ there is an $x\in X$
satisfying $d_Y(F(x),y)\le C$.

The map $F$ is a \emph{quasi-isometry} if both $\rho_-$ and $\rho_+$ can be chosen to be 
affine. It is a coarse (quasi-isometric) \emph{embedding} if it is a coarse (quasi-isometric)
equivalence with a subset of $Y$.

Two coarse maps $F,F':X\to Y$ are said to be \emph{close} if $d(F(x),F'(x))\le C$ for
some constant $C>0$ and every $x\in X$. 
\end{definition}
Since we are only concerned with the asymptotic behavior, we can without loss of generality assume that
$\rho_-$ is strictly increasing and thus has an inverse.
Let $F:X\to Y$ be a coarse embedding. Define
$$F_*\left( \sum c_x x\right)=\sum c_xF(x).$$
If $c\in C_i^f(X)$ then  
$$\vert c_{F(x)}\vert \le Kf(\vert x\vert )\le K f(\rho_-^{-1}(\vert F(x)\vert)),$$
since $f$ is non-decreasing, where $K$ depends 
on $f$ and $F$ only. Thus we obtain a map on the level of chains, 
$$F_* :C_i^f(X)\to C_i^{f\circ\rho_-^{-1}}(Y).$$
In particular, if $F$ is a quasi-isometric embedding, we have the induced map
$F_*:C_i^f(X)\to C_i^f(Y)$.
We denote also by $F_*$ the induced map on homology,
$$F_*:H_i^f(X)\to H_i^{f\circ\rho_-^{-1}}(Y).$$ 
\begin{proposition}
Let $F, F':X\to Y$ be quasi-isometric embeddings which are close. Then 
$F_*, F_*':C_i^f(X)\to C_i^f(Y)$ are chain homotopic.
\end{proposition}

\begin{proof}
The proof follows the one in  \cite{block-weinberger}. 
If $F$ and $F'$ are close then the map $\set{F,F'}:X_1\times \set{0,1}\to Y$
is a coarse map, so we need to show that $i_0,i_1:X\to X\times\set{0,1}$ are chain homotopic.
Let $H:C_i^f(X)\to C_{i+1}^f(X\times\set{0,1})$ be defined as the linear
extension of 
$$H(x_0,\dots,x_i)=\sum_{j=0}^i(-1)^j((x_0,0),\dots,(x_j,0),(x_j,1),\dots,(x_i,1)).$$ Then $\partial H+H\partial=i_{1*}-i_{0*}$.
\end{proof}

\begin{corollary}\label{corollary : QI invariance}
$H_*^f$ is a quasi-isometry invariant, i.e.
for a fixed $f$ and a quasi-isometry $F:X\to Y$ we have $H_n^f(X)\cong H_n^f(Y)$ for every 
$n\in \NN$.
\end{corollary}

\subsection{Vanishing of the fundamental class} 
The phenomenon that we want to explore is that the vanishing of the fundamental class 
$$[\Gamma]=\sum_{x\in \Gamma} [x]$$ 
in $H_0^f(\Gamma)$ for $f$ with a certain
growth type has geometric consequences. The most interesting case is when the space in question is a Cayley graph of a finitely generated group. Theorem \ref{theorem : linear preimage via the spread tail construction} shows that in this case the fundamental class vanishes when the growth type is at least linear (see also \cite{zuk}). On the other hand, if $f$ is a constant function
and  $H^f_0(\Gamma)=H^{\text{uf}}_0(\Gamma)$ then vanishing of the fundamental class $[\Gamma]$
in $H^{\text{uf}}_0(\Gamma)$ is equivalent to non-amenability of $\Gamma$
\cite[Theorem 3.1]{block-weinberger}. Our philosophy is that if $H^{\text{uf}}_0(\Gamma)$ is not zero (i.e.\ $\Gamma$ is amenable), then 
the vanishing of the fundamental class in $H^f_0(\Gamma)$ for an unbounded $f$
should quantify ``how amenable'' $\Gamma$ is, in terms of the growth of $f$: the
slower the growth of $f$, the ``less amenable'' the group $\Gamma$ is.

In the uniformly finite homology vanishing of the fundamental class is equivalent to vanishing of the $0$-th homology group. However, this is not expected to happen in general for the 
controlled homology. We record the following useful fact.

\begin{lemma}\label{lemma : bounded below chains vanish then [G] vanishes}
Let $c=\sum_{x\in X}c_x[x]\in C_0^f(X)$ be such that $c_x\ge C>0$ for some $C$. If 
$[c]=0$ in $H_0^f(X)$ then $[X]=0$ in $H_0^f(X)$.
\end{lemma}

\begin{proof}
First we will show that under the assumption there exists $c'=\sum c'_x[x]$ such that $c_x\in \NN\setminus \set{0}$ and such that $c=\partial \psi$ where $\psi\in C_1^f(X;\ZZ)$.
Denote $N=\sup_{x\in X}\# B(x,\mathscr{P}(\phi))$. Given $c$ and $\phi$ such that 
$\partial \phi=c$ take $\kappa>0$ sufficiently large to guarantee $\kappa C- N\ge 1$. We have 
$\partial (\kappa\phi)=\kappa c$. 
Now define
$$\psi(x,y)=\left\{
\begin{array}{ll}
\lceil \kappa \phi(x,y)\rceil&\text{ if } \phi(x,y)\ge 0 \\
\lfloor \kappa\phi(x,y) \rfloor&\text{ if } \phi(x,y)< 0
\end{array}\right.$$
where $\lceil\cdot\rceil$ and $\lfloor\cdot\rfloor$ 
are the "ceiling" and "floor" functions respectively.
Then $\psi\in C_1^f(X;\ZZ)$ and 
\begin{eqnarray*} 
\partial \psi(x) & \ge&\sum_{y\in X}\kappa\phi(y,x)-1\\
&\ge&\partial \kappa\phi(x)- N\\
&\ge&1.
\end{eqnarray*}
Now using the technique
from \cite[Lemma 2.4]{block-weinberger}, for every $x\in X$ we can
construct a ``tail'' $t_x\in C_1^f(X)$ such that $\partial t_x=[x]$. To build $t_x$ 
note that since $\partial\psi(x)\ge 1$ then there must exist $x_1\in X$ such that the
coefficient of $\psi(x,x_1)\ge 1$. Then, since $\psi(x_1,x)\le -1$ and $\partial \psi(x_1)\ge 0$
there must exist $x_2$ such that $\psi(x_2,x_1)\ge 1$, and so on. 
Then $t_x=[x,x_1]+[x_1,x_2]+\dots$. We apply the same 
procedure to $\psi-t_x$ and continue inductively. Since we are only choosing simplices that  
appear in $\psi$ we have
$\sum_{x\in X}t_x\in C^f_1(X)$. By construction $\partial (\sum_x t_x)=1_X$.  
  \end{proof}

\section{An explicit linear 1-chain}\label{section : an explicit 1-chain}

In this section we prove the first of the main results of this paper and explain 
why we are mainly interested in $f$ at most
linear.
One picture conveyed in \cite{block-weinberger} for vanishing of the fundamental class in 
$H_0^{\rm uf}(\Gamma)$ was the one of an \emph{infinite Ponzi scheme} or \emph{tails} 
(as $t_x$ constructed above). More precisely, the vanishing is equivalent to the existence 
of a collection of tails $t_x\in C_1^{\rm uf}(\Gamma)$ of the form
$[x,x_1]+[x_1,x_2]+[x_2,x_3]+\dots$ for each point $x\in\Gamma$, so that $\partial t_x=[x]$, 
and such that for any chosen radius, the number of tails passing through any ball of that 
radius is uniformly bounded. For amenable groups this is impossible to arrange. However, 
for any finitely generated group, we can always arrange \emph{some} kind of scheme, where 
the number of tails passing through a ball will be controlled by an unbounded function. Such 
escape routes to infinity are also known as \emph{Eilenberg swindles}.

The following  example explains why a linear control might be sufficient in general. 
On $\ZZ$, the
1-chain $\sum_{n\in \ZZ}n[n,n+1]$ has linear growth and its boundary is the
fundamental class $[\ZZ]=\sum_{x\in \ZZ}[x]$.  If $\Gamma$ has a
quasi--isometrically embedded cyclic subgroup then we can, roughly speaking,
just take the above 1-chain on every coset to bound the fundamental class of $[\Gamma]$.
In general however, as various solutions to the Burnside problem show, an infinite group
does not have to have infinite cyclic subgroups at all. Nevertheless, an infinite finitely
generated group  still has a lot of infinite geodesics,  which can play
the role of the cyclic subgroup.

\begin{theorem}\label{theorem : linear preimage via the spread tail construction} 
Let $\Gamma$ be a finitely generated infinite group. 
Then $[\Gamma]=0$ in
$H^{\mathrm{lin}}_0(\Gamma)$.
\end{theorem}

\begin{proof}
  Take a group $\Gamma=\langle S\rangle$ with a finite symmetric
  generating set $S$.  The notation we shall use is compatible with
  the one in \cite{zuk}, modulo the fact that we use the
  left-invariant metric induced from $S$, while in that paper the
  right-invariant metric is used.  Let $g_0$ be a bi-infinite geodesic
  through the identity $e\in \Gamma$, let $\G=\{\gamma g_0\mid \gamma\in \Gamma\}$
  be a left-translation invariant set of parametrized geodesics (i.e.\
  geodesics with a distinguished point). Let $G\subset \G$ be the set
  of all geodesics from $\G$ passing through $e$. We say that a subset
  $H$ of $\G$ is measurable, if it is a subset of a set of the form
  $\delta_1G\cup\dots\cup \delta_kG$ for some
  $\delta_1,\dots,\delta_k\in\Gamma$. Denote by $\F$ the set of
  measurable subsets of $\G$.  \.Zuk's construction \cite[Section 3.2]{zuk}, applied to a group 
  with a left-invariant metric, produces a finitely additive
  measure $\varphi$ on $\F$, which is left-invariant, and for which
  $\varphi(G)=1$.

  Let us remark that in general \.Zuk makes use of the invariant mean
  on $\ZZ$. However, if one can choose $g_0$ in such a way that
  $N=\#\{\gamma g_0\mid \gamma\in g_0\}<\infty$ (now we consider the
  geodesics $\gamma g_0$ as unparametrized), then the construction of
  such a measure greatly simplifies; see \cite[Question
  1]{zuk}. Indeed, then one can use just unparametrized geodesics and
  set $\varphi(\{\gamma g_0\})=\frac1N$. 

Note that given any path $s\subset \Gamma$, we can think of it as a $1$-chain
in $C^{\text{uf}}_1(\Gamma)$ with propagation $1$: $s$ is just a sequence of
points (finite, infinite or bi-infinite), e.g.\
$(\dots,\gamma_{-1},\gamma_0,\gamma_1,\dots)$, and we view $s$ as the chain
$\sum_{n\in\ZZ}[\gamma_n,\gamma_{n-1}]\in C^{\text{uf}}_1(\Gamma)$.

On each $h\in\G$, let us distinguish another point $p(h)\in h$, which
realizes the distance from $e$ to the set $h$ (if there are more such points,
just choose one arbitrarily). Furthermore, for each parametrized geodesic
$h\in\G$ containing $\delta$, choose the subray $\vec s_\delta(h)$ of $h$
which begins at $\delta$ and does not contain $p(h)$ (if $\delta=p(h)$,
choose one of the rays arbitrarily). If $h$ does not contain $\delta$, just
put $\vec s_\delta(h)=0$.

For each $\delta\in \Gamma$, we would like to define a ``spread tail''
$t_\delta\in C^{\text{uf}}_1(\Gamma;\RR)$ as
$$t_\delta=\sum_{h\in\G}\varphi(\{h\})\vec s_\delta(h),$$
but if $N$ is not finite, then $\varphi(\{h\})=0$. Instead, we define $t_\delta$ as follows: for each edge 
$(\gamma,\gamma s)$ ($s\in S$) in the Cayley graph of $\Gamma$,
define the set of geodesics
$$A_\delta(\gamma,\gamma s)=\{h\in\G\mid \delta\in h\text{ and }
 (\gamma,\gamma s)\subset\vec s_\delta(h)\text{ preserving the direction}\}.$$
Note that $A_\delta(\gamma,\gamma s)\subset \delta G$ is measurable, with
measure $\leq 1$. Now define 
$$
t_\delta=\sum_{\gamma\in\Gamma, s\in S}
  \varphi(A_\delta(\gamma,\gamma s))[\gamma s,\gamma]
$$
Obviously, the cycle $t_\delta$ has coefficients uniformly bounded by
$1$.\\\\

\begin{figure}[htp]
\centering
     \includegraphics{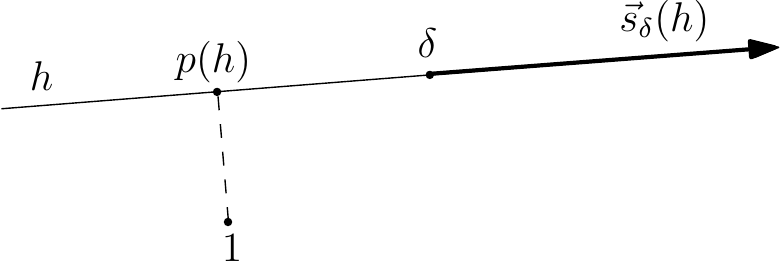}
\caption{A tail}\label{figure : tail}
\end{figure}

\begin{lemma}  $\partial t_\delta=[\delta]$
\end{lemma}
\begin{proof} Consider any point $\gamma\in \Gamma\setminus\{\delta\}$, and
note that there are only finitely many edges $(\gamma,\gamma s)$,
$s\in S$ going out of the vertex $\gamma$. The coefficient of
$[\gamma]$ in $\partial t_\delta$ is
$$
\sum_{s\in S}\varphi(A_\delta(\gamma,\gamma s))-
 \sum_{s\in S}\varphi(A_\delta(\gamma s,\gamma))=
 \varphi\Bigl(\bigcup_{s\in S}A_\delta(\gamma,\gamma s)\Bigr)-
 \varphi\Bigl(\bigcup_{s\in S}A_\delta(\gamma s,\gamma)\Bigr)
$$
Any ray $\vec s_\delta(h)$ going through $\gamma$ contributes $h$ to each of
the sets $\bigcup_{s\in S}A_\delta(\gamma,\gamma s)$ and
$\bigcup_{s\in S}A_\delta(\gamma s,\gamma)$. On the other hand, if the ray
$\vec s_\delta(h)$ does not pass through $\gamma$, $h$ is not contained in
any of these sets. Hence, the expression in the above display is
$0$. Finally, looking at the coefficient of $[\delta]$ in $\partial
t_\delta$, observe that the rays $s_\delta(h)$ only go away from $\delta$ and
all the geodesics $h\in\G$ passing through $\delta$ are used. Consequently,
the coefficient is
$\varphi(\bigcup_{s\in S}A_\delta(\delta,\delta s))=\varphi(\delta G)=1$.
\end{proof}
\smallskip

Now let $\psi=\sum_{\delta\in \Gamma}t_\delta$. Clearly $\psi$ has
propagation $1$ and $\partial \psi = \sum_{\delta\in \Gamma}[\delta]$. The
proof of the theorem is finished by the following Lemma.
\end{proof}

\begin{lemma}$\psi\in C_1^{\lin}(\Gamma;\RR)$, i.e.\ $\psi$ has linear
growth.\end{lemma}

\begin{proof}The idea behind the argument is the following: take
$\delta\in\Gamma$ and a geodesic $h\in\G$, which passes through $\delta$. We
need to count those points $\beta$ on $h$, for which the chosen ray $\vec
s_\beta(h)$ passes through $\delta$. It is easy to see that these are precisely
the points on $h$ between $p(h)$ and $\delta$. By the triangle inequality and
the choice of $p(h)$, we have that the number $K(\delta,\gamma)$ of those
points is at most $\vert p(\gamma)\vert+\vert \delta\vert \leq 2\vert\delta\vert$, where by our convention $\vert \gamma\vert$ denotes the length of the element $\gamma$. It
follows that the coefficient of edges ``attached" to $\delta$ is at most
$$
\sum_{h\in \delta G}\varphi(\{h\})K(\delta,\gamma)
\leq \sum_{h\in \delta G}2\varphi(\{h\})\vert\delta\vert=2\varphi(\delta G)\vert \delta\vert=
2\vert\delta\vert.
$$
Of course, the problem is again that all $\varphi(\{h\})$ can vanish and so we
need to adjust the argument.

Pick any edge $(\gamma, \gamma s)$ in the Cayley graph of $\Gamma$ and denote
by $c_{[\gamma,\gamma s]}$ the coefficient of $[\gamma,\gamma s]$ in
$\psi$. From the definition, we have $c_{[\gamma,\gamma
  s]}=\sum_{\delta\in\Gamma}\varphi(A_\delta(\gamma,\gamma s))$. The task is to
prove that $c_{[\gamma,\gamma s]}$ depends linearly on $\vert\gamma\vert$.

We are going to further split each $A_\delta(\gamma,\gamma s)$. Let
$P(\gamma,\gamma s)$ be the collection of all finite (non-parametrized)
geodesic paths in $\Gamma$, which are subpaths $[p(h),\gamma]$ of some ray
$\vec s_\delta(h)$ containing $(\gamma,\gamma s)$, $\delta\in \Gamma$,
$h\in\G$. The collection $P(\gamma,\gamma s)$ is finite. Indeed, for any $\vec
s_\delta(h)$ containing $(\gamma,\gamma s)$ we have that $\vert p(h)\vert\leq
\vert \gamma\vert $ and the assertion follows from bounded geometry.

For each $\delta\in\Gamma$ and $a\in P(\gamma,\gamma s)$, we
denote 
\begin{align*}
B(a)&=\{h\in \G\mid \vec s_\delta(h)\text{ begins with }a\},\\
B(\delta,a)&=\{h\in A_\delta(\gamma,\gamma s)\mid \vec s_\delta(h)\text{
  begins with }a\}.
\end{align*}
All these sets are measurable, since they are subsets of
$\gamma G$. Furthermore $A_\delta(\gamma,\gamma s)=\bigsqcup_{a\in
  P(\gamma,\gamma s)}B(\delta,a)$. Note that $B(\delta,a)$ is either empty
(when $\delta\not\in a$), or equal to $B(a)$ (if $\delta\in a$). For a given
$a\in P(\gamma,\gamma s)$, the number of $\delta$ for which $\delta\in a$
is bounded by $\text{length}(a)=d(\gamma,p(h))\leq
\vert \gamma\vert+\vert p(h)\vert \leq 2\vert \gamma\vert$ (where $h$ is arbitrary element of
$B(a)$). Putting this information together, we obtain an estimate
\begin{eqnarray*}
c_{[\gamma,\gamma s]}&=&\sum_{\delta\in\Gamma}\varphi(A_\delta(\gamma,\gamma
s))\\\\
&=&\sum_{\delta\in\Gamma}\sum_{a\in P(\gamma,\gamma
s)}\varphi(B(\delta,a))\\\\
&=&\sum_{a\in P(\gamma,\gamma s)}\sum_{\delta\in
a}\varphi(B(a))\\\\
&\leq&2\vert \gamma\vert \cdot\left(\sum_{a\in P(\gamma,\gamma s)}\varphi(B(a))\right)\\\\
&\leq&2\vert \gamma\vert \varphi(\gamma G)\\\\
&=&2\vert \gamma\vert.
\end{eqnarray*}
For completeness, note that the sum over $\delta\in \Gamma$ in the above
display has in fact finitely many non-zero terms. The last inequality uses the
fact that $B(a)$'s are disjoint for different $a$'s.
\end{proof}

Motivated by his results on a geometric version of the von Neumann conjecture
 \cite{whyte} Whyte asked whether a geometric version of the Burnside problem 
problem has a positive solution. The geometric Burnside problem, as formulated in 
\cite{whyte}, asks whether every infinite group admits a free translation action of $\ZZ$, where
a translation is understood as a bijective map which is close to the identity. 
When a group has a cyclic subgroup then clearly the left translation action by this subgroup 
is such an action. 

Note that if the answer would be affirmative and, additionally, the orbits would be 
undistorted in the group $G$, 
then such undistorted geometric Burnside problem implies 
Theorem \ref{theorem : linear preimage via the spread tail construction}
by simply copying the 1-chain $\sum n[n,n+1]$ onto every orbit. 
This means that Theorem \ref{theorem : linear preimage via the spread tail construction}
in fact gives an affirmative answer to a weak, homological version of the Burnside problem.
In other words, from the point of view of coarse homology, every infinite finitely generated 
group behaves as if it had an undistorted infinite cyclic subgroup.
To the best of our knowledge this is the strongest existing positive result in this direction.

\begin{remark}\normalfont
The above proof depends on the fact that a group is a very symmetric object, and thus there are generally many directions to escape to infinity.
This is certainly not the case for more general bounded geometry 
metric spaces or manifolds, as the following example shows. Consider the set
$X=\NN\times\{0\}\,\cup\,\left(\bigcup_{i\in\NN}\{n\}\times\{0,1,\dots,a_{n}\}\right)\subseteq 
\ZZ\times \ZZ$,
where $a_n$ is a sequence of natural numbers which increases to infinity.
The metric on $X$ (the ``lumberjack metric'') is defined by
$$d\big(\,(x,y),(x',y')\,\big)=\vert y\vert+\vert x-x'\vert+\vert y'\vert.$$
In this metric space there is only one way to infinity, that is via the horizontal 
line, hence the ``tails'' that constitute a cycle which kills the fundamental class have to escape through this one route. By controlling the growth of the sequence $a_n$ one can easily
impose any lower bound on the growth of any $1$-chain which kills the fundamental class $[X]$.
\end{remark}

\section{$H^f_0(\Gamma)$ and isoperimetric inequalities}\label{section : homology and inequalities}

There is number of isoperimetric inequalities which one can study on a 
finitely generated group on a bounded geometry space $X$. The one we are interested in, studied
 in \cite{erschler,zuk}, is the isoperimetric inequality of the form 
\begin{equation}\tag{$I^{\,f}_{\partial}$}
\label{equation : Zuk-type isoperimetric inequality}
\#A\le C \sum_{x\in \partial A} f(\vert x\vert)
\end{equation}
for all finite sets $A\subset X$, where $\partial A$ is the boundary of $A$ (see below). We will say that
\emph{$X$ satisfies inequality (\ref{equation : Zuk-type isoperimetric inequality})} if there 
exists a constant $C>0$ such that 
(\ref{equation : Zuk-type isoperimetric inequality}) holds for every finite set $A\subseteq X$.  
When $f$ is constant this isoperimetric
inequality is equivalent to non-amenability of the group.

We set
\begin{align*}
\partial A &= \setd{g\in \Gamma}{d(g,A)=1\text{ or }d(g,\Gamma\setminus A)=1},\\
\partial^{\rm e} A &= \setd{(g,h)\in A\times (\Gamma\setminus A)}{d(g,h)=1} \cup
\setd{(g,h)\in (\Gamma\setminus A)\times A}{d(g,h)=1}.
\end{align*}
The first lemma reformulates the inequality \eqref{equation : Zuk-type
isoperimetric inequality} in terms of functions. This form will be much
more convenient to work with in connection to coarse homology. Recall 
that for an $n$-simplex $\overline{x}$ we denote by $\vert \overline{x}\vert$ the 
distance of $\overline{x}$ from the fixed point $\overline{e}$. In particular, if the simplex is 
an edge $(x,y)$, we use the notation $\vert(x,y)\vert$.

\begin{lemma}\label{lemma : equivalent forms of isoperimetric inequality}
Let $X$ be a metric space of bounded geometry. The following two conditions are equivalent
\begin{enumerate}
\renewcommand{\labelenumi}{(\alph{enumi})}
\item $X$ satisfies inequality (\ref{equation : Zuk-type isoperimetric inequality}),\\
\item the inequality 
\begin{equation}\tag{$I^{\,f}_{\nabla}$}
\label{equation : functional version of Zuk's ineq}
\sum_{x\in X}\vert\eta(x)\vert\ \le\ D\left(\sum_{x\in X}\sum_{y\in B(x,1)} \left\vert\eta(x)-\eta(y)\right\vert f(\vert (x,y)\vert)\right)
\end{equation}
holds for every finitely supported function
$\eta:X\to \RR$ and a constant $D>0$.
\end{enumerate}
\end{lemma}
\begin{proof}
(b) implies (a) follows just by applying the inequality to
$\eta=1_A$, and using the property \eqref{equation : f has the same
growth under shifts}. To prove the other direction we use a standard
``co-area'' argument. It is enough to 
restrict to the case $\sum_{x\in X}\vert \eta(x)\vert=1$ and we first consider $\eta\ge 0$.
By density arguments it is sufficient prove the
claim for functions of $\ell_1$-norm 1 which take values in sets of the form $\set{\frac{i}{M}}_{i\in\NN}$
(where $M$ is chosen for each function separately).
For such $\eta$ we denote $A^{(i)}=\setd{x\in X}{\eta(x)> \frac{i}{M}}$.
Then we can write $\eta(x)=\frac1M\sum_{i\in\NN}1_{A^{(i)}}(x)$ and $\sum_{i\in \NN} \# A^{(i)}=M$. Furthermore,
\begin{eqnarray*}
\sum_{d(x,y)\le 1} \left\vert \eta(x)-\eta(y)\right\vert f(\vert (x,y)\vert)&=&
\frac{1}{M} \sum_{i\in \NN}\left(\sum_{x\in X}\sum_{y\in B(x,1)}\left\vert 1_{A^{(i)}}(x)-1_{A^{(i)}}(y)\right\vert f(\vert (x,y)\vert)\right)\\\\
&\ge&\sum_{i\in \NN} \frac{1}{M}\left( \sum_{(x,y)\in \partial^e A^{(i)}} f(\vert (x,y)\vert)
\right)\\\\
&\ge&C \sum_{i\in \NN} \frac{1}{M}\left( \sum_{x\in \partial A^{(i)}} f(\vert x\vert)
\right)\\\\
&\ge&\sum_{i\in\NN} \frac{1}{M}\, \#A^{(i)}\\\\
&=&1.
\end{eqnarray*}
For a general $\eta$ apply the above inequality to $\vert \eta\vert$ together with the triangle
inequality to obtain

$$\sum_{x\in X}\vert\eta(x)\vert\le\sum_{d(x,y)\le 1}\Big\vert\, \vert \eta(x)\vert-\vert\eta(y)\vert\,\Big\vert
f(\vert(x,y)\vert) \le\sum_{d(x,y)\le 1}\vert \eta(x)-\eta(y)\vert f(\vert(x,y)\vert).$$
\end{proof}

The next theorem gives a homological description of the inequality
\eqref{equation : Zuk-type isoperimetric inequality} and is the second of the main
results of this paper. It can be
understood as providing a passage from knowing some information ``on
average'' (the inequality) to precise, even distribution of coefficients
of $\psi\in C_1^f(X)$ which satisfies $\partial \psi=[X]$.

\begin{theorem}\label{theorem : vanishing iff inequality holds}
Let $X$ be a bounded geometry quasi-geodesic metric space. The following conditions are equivalent:
\begin{enumerate}
\renewcommand{\labelenumi}{(\Alph{enumi})}
\item $X$ satisfies inequality (\ref{equation : Zuk-type isoperimetric inequality}),
\item $[X]=0$ in $H_0^f(X)$.
\end{enumerate}
\end{theorem}

Some comments are in order before we prove the theorem. 
In the uniformly finite case \cite{block-weinberger} the proof of one implication  
relies on the Hahn--Banach theorem, which allows to build a functional
distinguishing between boundaries and the fundamental class. However,
our chain groups have a topology which does not allow to conveniently generalize this argument.
There is another way to view the equivalence \emph{$X$ non-amenable if and only if
$H_0^{\uf}(X)=0$}, which we roughly sketch here (with the notation from the proof of Theorem \ref{theorem : vanishing iff inequality holds} below).

Non-amenability of $X$ is equivalent to the inequality $\|\eta\|_1\leq C\|\delta \eta\|_1$, where $\delta:\ell_1(X)\to\ell_1^{\,\mathrm{ch}}(N\Delta)$ is defined by $\delta\eta(x,y)=\eta(y)-\eta(x)$, $N\Delta$ is the 1-neighborhood of the diagonal in $X\times X$ and $\ell_1^{\,\mathrm{ch}}(N\Delta)$ denotes the 
absolutely summable $1$-chains of propagation at most 1.
As a linear map between Banach spaces $\delta$ is continuous and, by the above inequality, also topologically injective. Since the topological dual of $\ell_1^{\,\mathrm{ch}}(N\Delta)$ is essentially 
$C_1^{\uf}(X)$, this is further equivalent, by duality, to surjectivity of the adjoint map $\widetilde{\partial}:C_1^{\uf}(X)\to C_0^{\uf}(X)$  (the latter space can be simply viewed as $\ell_\infty(X)$) which is, up to a multiplicative constant, the same as our differential. This is however exactly the vanishing
of the $0$-dimensional homology group.

It is this point of view which we use to prove Theorem \ref{theorem : vanishing iff inequality holds}. 
However, if $f$ is not constant, the maps $\partial$ and $\delta$ are not continuous and thus
do not respond  directly to the above argument.
Before embarking on the proof, we record a simple but necessary fact.

\begin{lemma}\label{lemma : cutting down propagation of preimages}
Let $X$ be a bounded geometry quasi-geodesic metric space. Then $[X]=0$ 
in $H_0^f(X)$ if and only if $1_X=\partial \psi$, where $\psi\in C_1^f(X)$ and
$\mathscr{P}(\psi)\le 1$.
\end{lemma}
\begin{proof}
To prove the nontrivial direction replace every edge $[x,y]$ with  $d(x,y)>1$ and
a non-zero coefficient $a(x,y)$ by the chain
$a(x,y)\sum [x_i,x_{i+1}]$ where the $x_i$ are given by the quasi geodesic condition,
so that $d(x_i,x_{i+1})\le 1$. 
\end{proof}

\begin{proof}[Proof of Theorem \ref{theorem : vanishing iff inequality holds}] 
We equip $X\times X$ with the measure $\nu(x,y)=f(\vert (x,y)\vert)$.
Let $N\Delta=\setd{(x,y)\in X\times X}{d(x,y)\le 1}$ denote the 1-neighborhood of the diagonal. We consider the linear space
$$\ell_{\infty}^{\,f}(N\Delta)=\setd{ \psi:N\Delta\to \RR}
{\sup_{(x,y)\in N\Delta}\dfrac{\vert \psi(x,y)\vert}{f(\vert(x,y)\vert)}< \infty}$$
with the norm
$$\Vert \psi\Vert_{\infty}^f=\sup_{(x,y)\in N\Delta}\dfrac{\vert\psi(x,y)\vert}{f(\vert (x,y)\vert)}.$$ 
Denote also
$$\ell_1(N\Delta,\nu)=\setd{\psi:N\Delta\to \RR}{\sum_{(x,y)\in N\Delta}
\vert\psi(x,y)\vert\nu(x,y)<\infty}$$
and equip it with the norm
$$\Vert \psi\Vert_{1,\nu} = \sum_{(x,y)\in N\Delta}\vert \psi(x,y)\vert \nu(x,y).$$
For $\psi,\phi:X^n\to\RR$, we denote
$$\langle \psi,\phi\rangle=\sum_{\overline{x}\in X^{n+1}}\psi(\overline{x})\phi(\overline{x}),$$
whenever this expression makes sense. The topological dual of $\ell_1(N\Delta,\nu)$ with respect to this pairing is 
$\ell_{\infty}^{\,f}(N\Delta)$.
Let $\finsup$ denote finitely supported functions on $X$ and define $\delta: \finsup\to \ell_1(N\Delta,\nu)$ to be the map defined by 
$$\delta\eta(x,y)=\eta(y)-\eta(x), \quad \text{for }d(x,y)\leq1.
$$
Then define a linear operator $\widetilde{\partial}:\ell_{\infty}^{\,f}(N\Delta)\to \ell_{\infty}^f(X)$
by setting
$$\widetilde{\partial}\psi(x)=\sum_{y\in B(x,1)}\ \psi(y,x)-\psi(x,y).$$
On chains,  $\widetilde{\partial}$ is algebraically dual to $\delta$ with the above pairings, i.e. 
$$\langle \eta,\widetilde{\partial}\psi\rangle=\langle \delta\eta,\psi\rangle$$
as finite sums for all $\eta\in \finsup$ and $\psi\in C^{\,f}_1(X)$ such that $\mathscr{P}(\psi)\le 1$. Note that $\widetilde{\partial}$
is also a linear extension of $2\partial$ from $C_{1}^f(X)$ to $\ell_{\infty}^{\,f}(N\Delta)$ for such $\psi$.\\

\noindent (B)  $\Longrightarrow$ (A).  Assume that there exists a chain $\psi\in C^{\,f}_1(X)$ such that $1_{X}=\partial \psi$ and $\mathscr{P}(\psi)=1$ (by 
Lemma \ref{lemma : cutting down propagation of preimages}).
For a non-negative function $\eta\in \finsup$ we have

\begin{eqnarray*}
\Vert \eta\Vert_1&=&\sum_{x\in X}  \eta(x)\\\\
&= &\sum_{x\in X} \eta(x)\cdot \partial\psi(x)\\\\
&= &\frac{1}{2}\sum_{(x,y)\in X\times X}\delta\eta(x,y) \cdot \psi(x,y)\\\\
&\le&\frac{1}{2}\sum_{d(x,y)\le 1} \vert \eta(x)-\eta(y)\vert\ \vert\psi(x,y)\vert\\\\
&\le&C\ \sum_{d(x,y)\le 1} \vert \eta(x)-\eta(y)\vert\ f(\vert (x,y)\vert).
\end{eqnarray*}
By Lemma \ref{lemma : equivalent forms of isoperimetric inequality}, we are done.\\

\noindent (A) $\Longrightarrow$ (B). By Lemma \ref{lemma : cutting down propagation of
preimages} there
 is no loss 
of generality by restricting to functions of propagation 1.
Let $\delta\finsup$ denote the image of $\finsup$ under $\delta$.
By Lemma \ref{lemma : equivalent forms of isoperimetric inequality}
we rewrite the inequality (\ref{equation : Zuk-type isoperimetric inequality}) 
in the functional form as (\ref{equation : functional version of Zuk's ineq}) i.e.,
 $$\sum \vert \eta(x) \vert\le C\left(\sum_{d(x,y)\le 1}\vert \eta(x)-\eta(y)\vert f(\vert (x,y)\vert)\right)$$ 
 for every $\eta\in \finsup$.
If we interpret this inequality in terms of the norms of elements of $\finsup$
and $\delta\finsup$ it reads
$\Vert \eta\Vert_1\le C\Vert \delta\eta\Vert_{1,\nu}$.
The map 
$\delta: \finsup\to \delta\finsup$ is a bijection, thus there exists an inverse 
$\delta^{\,-1}:\delta\finsup\to \finsup$.
The inequality implies that this inverse is continuous. We extend it to 
a continuous map 
$$\delta^{\,-1}:\overline{\delta\finsup}\to \ell_1(X),$$
where $\overline{\delta\finsup}$ is the norm closure of $\delta\finsup$  in $\ell_1(N\Delta,\nu)$.
This map  induces a continuous adjoint map 
$(\delta^{\,-1})^*:C_0^{\uf}(X)\to {\overline{\delta\finsup}^{\,*}}$
between the dual spaces, which satisfies
$$\left\langle \delta\eta, (\delta^{\,-1})^* \zeta \right\rangle= \langle \eta,\zeta\rangle$$
for $\eta\in \finsup$ and $\zeta\in C_0^{\uf}(X)$.
We thus have the following diagrams, where the top one is dual to the bottom one:
\baselineskip=10pt

$$
\begin{diagram}
\ell_{\infty}^{\,f}(N\Delta)		& &	\\
\dOnto^{i^*}	&	&\\
\overline{\delta\finsup}^{\,*}& \lTo^{\ \ \ \ (\delta^{\,-1})^*\ \ \ \ }& C_0^{\uf}(X)
\end{diagram}
$$
\newline
$$
\begin{diagram}
\overline{\delta\finsup}&\pile{ \rTo^{\ \ \ \ \ \ \delta^{-1}\ \ \ \ \ \ \ } \\ \lDotsto_{\delta\ \ \ \ \ \ \ }}& \ell_1(X)\ \ \ \ \\
\dInto^{i}&&\\
\ell_1(N\Delta,\nu)
\end{diagram}\\
$$
\baselineskip=14pt
\newline
\noindent Here $i$ is the natural injection, $i^*$ is the restriction and the dashed arrow
denotes a discontinuous map which is densely defined on $\ell_1(X)$. 

Note now that if $\eta\in\finsup$ and $\psi\in \ell_{\infty}^f(X)$ is a 1-chain, the duality 
$\langle \widetilde{\partial} \psi,\eta\rangle=\langle \psi,\delta \eta\rangle$ says that 
$\psi$ is determined on $\delta\finsup$ only. This allows to construct a
preimage of the fundamental class in the following way.
Take $\phi$ to be any element in $\ell_{\infty}^{\,f}(N\Delta)$ such that 
$$i^{\,*}\phi=(\delta^{\,-1})^*1_{X},$$
(for instance any extension of $(\delta^{\,-1})^*1_{X}$ to a functional on the whole $\ell_1(N\Delta;\nu)$
guaranteed by the Hahn-Banach theorem). 
Then $\phi$ is an element of 
$\ell_{\infty}^{\,f}(N\Delta)$ and might not a priori belong to $C_1^f(X)$. 
To correct this we anti-symmetrize $\phi$.
Denote the transposition $\phi^T(x,y)=\phi(y,x)$. We define
$$\psi=\phi-\phi^T.$$
Then $\psi^T=-\psi$ so that $\psi\in C_1^f(X)$ and 
we will now show that $\partial \psi=1_{X}$ in $C_0^{f}(X)$. 
For every $x\in X$ and its characteristic function $1_x$ we have

\begin{eqnarray*}
\widetilde{\partial}\psi(x)&=&\left\langle 1_x, \widetilde{\partial} \psi\right\rangle\\
&=&\left\langle \delta 1_x, \psi \right\rangle\\
&=&\left\langle\delta 1_x,\phi\right\rangle-\left\langle\delta 1_x,\phi^T\right\rangle\\
&=&\left\langle\delta 1_x,\phi\right\rangle+\left\langle\delta 1_x^T,\phi^T\right\rangle\\
&=&2\left\langle \delta 1_x,\phi\right\rangle\\
&=&2\left\langle  \delta 1_x, i^{\,*}\phi  \right\rangle\\
&=&2\left\langle \delta 1_x, (\delta^{\,-1})^*1_{X} \right\rangle\\
&=&2\left\langle  1_x,1_{X}\right\rangle\\
&=& 2.\\
\end{eqnarray*}
In the above calculations we used the fact that $\delta 1_x^T=-\delta 1_x$ and that $\langle \psi,\phi\rangle=
\langle \psi^T,\phi^T\rangle$. Finally since $2\partial =\widetilde{\partial}$ so $\partial\psi=1_X$.
\end{proof}

It is interesting to note that  the linear space of 1-chains of propagation
at most 1 is a complemented subspace of $\ell_{\infty}^{\,f}(N\Delta)$. Indeed,
the anti-symmetrizing map $P(\phi)=\phi-\phi^T$ is a bounded projection from
$\ell_{\infty}^f(N\Delta)$ onto that subspace.
Consequently our argument shows that actually for any $\eta\in C_0^{\uf}(X)=\ell_\infty(X)$,
and a lifting  $\phi_{\eta}\in\ell_{\infty}^f(N\Delta)$ of $(\delta^{\,-1})^*\eta$, the projection $P\phi_{\eta}$ satisfies $\partial P\phi_{\eta}=\eta$, so that
$P$ and $(\delta^{\,-1})^*$ together with the Hahn-Banach theorem are used to construct a right inverse to $\widetilde{\partial}$. In other words, we get the following

\begin{corollary}
For a bounded geometry metric space $X$ the following are equivalent:
\begin{enumerate}
\item $[X]=0$ in $H_0^f(X)$,
\item the homomorphism $\jmath_0^{\,f}:H_0^{\uf}(X)\to H_0^f(X)$, induced by inclusion of chains, is trivial. 
\end{enumerate}
\end{corollary}
\noindent The above fact was known in the uniformly finite case \cite[Proposition 2.3]{block-weinberger}. Another corollary of the
proof is a different proof, via the isoperimetric inequality, of Lemma \ref{lemma : bounded below chains vanish then [G] vanishes}.
We can also
identify $\overline{\delta\finsup}^*$ with $\ell_{\infty}^{\,f}(N\Delta)\big/ \operatorname{Ann}(\overline{\delta\finsup})$ up to an isometric isomorphism, where $\operatorname{Ann}(E)=\setd{\psi\in E^*}{\langle \psi,\eta\rangle=0 \text{ for }\eta\in E}$ is the annihilator of $E$.

\begin{remark}\normalfont
Note that if $f\equiv \mathrm{const}$ and we replace $\ell_1$ and $\ell_{\infty}$ by $\ell_p$ and $\ell_q$
respectively ($1< p,q< \infty$, $1/p+1/q=1$), then the above argument together with the fact that non-amenability is equivalent to
$\Vert \eta\Vert_q\le C\Vert \delta\eta\Vert_q$ for any $1\le q<\infty$, gives a different proof of a theorem 
in \cite{whyte} and \cite{elek},
namely
that vanishing of the  $0$-th $\ell_q$-homology group
characterizes non-amenability of metric spaces. 
\end{remark}

\begin{remark}\normalfont
The above proof gives a dual characterization of the best constant $C$ in the
isoperimetric inequality in terms of the distance from the origin to the affine subspace
$\partial^{-1}(1_X)$ in $\ell_{\infty}^f(X)$. In particular for $f\equiv 1$ this applies to 
the Cheeger constant of $X$.
\end{remark}

\begin{remark}\normalfont
Combining the result of \.{Z}uk \cite[Theorem 1]{zuk} with Theorem \ref{theorem : vanishing iff inequality holds} gives another proof of the Theorem \ref{theorem : linear preimage via the spread tail construction}, but of course the proof in the previous section has the advantage of being constructive 
while the above is only an existence statement.
\end{remark}

\section{Examples}\label{section : examples}

\subsection*{Wreath products} 
In \cite{zuk} \.{Z}uk proved the inequality (\ref{equation : Zuk-type isoperimetric inequality})
with $f(t)=t$ on any finitely generated group and asked if there are groups for
which $f$ can be chosen to be of slower growth. 
Examples of such groups were constructed by Erschler \cite{erschler}.
Recall that the (restricted) wreath product is defined as the semidirect product 
$$G\wr H=\left(\bigoplus_{h\in H} G\right)\rtimes H$$
where $H$ acts on $\bigoplus_{h\in H} G$ by translation of coordinates.
Using Erschler's results, together with the above characterization, we exhibit finitely generated groups for which $[\Gamma]=0$
in $H_0^f(\Gamma)$ with $f$ growing strictly slower than linearly. In \cite{erschler} Erschler 
showed that for a group of the form $F\wr \ZZ^d$, where $d\ge 2$ and $F$ is a non-trivial 
finite group, inequality 
(\ref{equation : Zuk-type isoperimetric inequality}) holds with $f(n)= n^{1/d}$.
This gives
\begin{corollary}
Let $\Gamma=F \wr \ZZ^d$ for a non-trivial finite group $F$. Then  $[\Gamma]=0$ 
in $H^{f}_0(\Gamma)$, where $f(n)=n^{1/d}$.
\end{corollary}
It is also possible to exhibit groups for which the fundamental
class vanishes in $H_0^f(\Gamma)$ for $f(n)= \ln n$.
\begin{corollary}
Let $\Gamma=F\wr(F\wr\ZZ)$ for a non-trivial finite group $F$. Then $[\Gamma]=0$ in
$H_0^f(\Gamma)$, where $f(n)=\ln n$.
\end{corollary}
Iteration of the wreath product leads to successively slower growing functions, see \cite{erschler}.


\subsection*{Polycyclic groups} Let us denote 
$$\rad F=\min \setd{r}{F\subseteq B(g,r), g\in \Gamma}.$$

\begin{definition}Let $\Gamma$ be an infinite, amenable group. We define the isodiametric profile
 $\isod_{\Gamma}:\NN\to \NN$ of $\Gamma$ by the formula
$$\isod_{\Gamma}(r)=\sup\ \dfrac{\#F}{\#\partial F},$$
where the supremum is taken over all finite sets $F\subset \Gamma$ with the property 
$\rad F\le r$.

\end{definition}
In other words, the isodiametric profile finds finite sets with the smallest boundary
 among 
those with prescribed diameter. It can be equivalently described as the smallest function 
$\isod$ such
that the inequality $ \dfrac{\# F}{\# \partial F}\le \isod\left(\rad F\right)$
holds for all finite subsets $F\subseteq \Gamma$.  It is easy to see that
$\isod$ is sublinear,
in the sense that $D(n)\le Cn$ for some $C>0$ and all $n>0$. Also, in some sense, $D$ is an 
inverse of the function $\mathrm{A}_X$ introduced in \cite{nowak-exactiso}.

\begin{proposition}
  Let $\Gamma$ be a finitely generated group. If $\Gamma$ 
  satisfies inequality (\ref{equation : Zuk-type isoperimetric inequality}) then there exists $C>0$ 
  such that   $$\isod\le Cf.$$
  In particular, the above estimate holds when the fundamental class of $[{\Gamma}]$
  vanishes in $H_0^{f}({\Gamma})$.
\end{proposition}

\begin{proof}
  Let $F\subset \Gamma$ be a finite subset. Translating to the origin we can assume that
  $F\subseteq B(e, \rad F)$. By inequality
  \eqref{equation : Zuk-type isoperimetric inequality} and monotonicity of
  $f$ we obtain
  $$
  \#F\leq C \sum_{x\in \partial F}f(|x|)\leq \#\partial F\cdot C f(\rad F),$$
  which proves the claim.
\end{proof}

It is well-known that for infinite polycyclic groups $D$ grows linearly \cite{pittet}. Thus we have
\begin{corollary}
Let ${\Gamma}$ be an infinite, polycyclic group. Then  the fundamental class $[{\Gamma}]$ 
vanishes in 
$H^f_0({\Gamma})$ if and only if $f$ is linear.
\end{corollary}

\section{Obstructions to weighted Poincar\'{e} inequalities}
In \cite{li-wang} the authors studied a weighted Poincar\'{e} inequality of the form
\begin{equation}\tag{$P_{\rho}$}\label{equation : continuous weighted Poincare inequality} 
\int_M \eta(x)^2 \rho(x)\ v\ \le \ C \int_M \vert \nabla\eta\vert^2\ v
\end{equation}
($v$ is the volume form) for $\rho>0$ and its applications to rigidity of manifolds. For the purposes
of \cite{li-wang} it is useful to know what $\rho$ one can choose since the rigidity
theorems of that paper hold under the assumption that the curvature is bounded below by $\rho$. 
Just as with $f$, we assume  throughout  
that $\rho$ is a function of the distance from a base point and that $\rho> 0$.

In this section we establish a relation between the inequalities 
(\ref{equation : Zuk-type isoperimetric inequality}) and 
(\ref{equation : continuous weighted Poincare inequality}). First we need an auxiliary notion.
Let $f:[0,\infty)\to [0,\infty)$ be non-decreasing. We will say that 
\emph{$f$ is a slowly growing function} if
 for every $\varepsilon>0$ there is a $t_0$
such that for every $t>t_0$ we have $f(t+1)\le  f(t)+\varepsilon$.  
Examples of such $f$ are all convex sublinear functions, for instance
$f(t)=t^{\alpha}$
for $\alpha<1$ and $f(t)=\ln t$. A sufficient condition is that $f$ is convex and
 $\lim_{t\to \infty} f'(t)=0$.
The main result of this section is the following.

\begin{theorem}\label{theorem : weighted Poincare implies vanishing} 
Let $f:[0,\infty)\to [0,\infty)$ be  a non-decreasing, slowly growing function. Let
$M$ be a compact manifold 
such that the universal cover satisfies the weighted
Poincar\'{e} inequality (\ref{equation : continuous weighted Poincare inequality}) with weight $\rho:\widetilde{M}\to \RR$ given by $\rho(x)=1/f(d(x,x_0))$ for 
every smooth compactly supported function $\eta:\widetilde{M}\to\RR$ and 
a fixed point $x_0\in \widetilde{M}$.
Then $[\Gamma]=0$ in $H_0^f(\Gamma)$ where $\Gamma=\pi_1(M)$.
\end{theorem}

To prove Theorem \ref{theorem : weighted Poincare implies vanishing}
we will use techniques from a classical paper of Brooks \cite{brooks-spectrum}. We first consider
a fundamental domain $F$ for the action of the fundamental group on $\widetilde{M}$ by
taking a smooth triangulation and choosing for each $n$-simplex $\Delta$ in $M$ a simplex  $\widetilde{\Delta}$ in 
$\widetilde{M}$,
which covers $\Delta$. The fundamental domain is the union of these simplices. 

\begin{lemma}
Let $\rho:\widetilde{M}\to \RR$ be constant when restricted to $\gamma F$ for every $\gamma\in \Gamma$.
If  the continuous weighted
Poincar\'{e} inequality (\ref{equation : continuous weighted Poincare inequality}) with weight $\rho$
holds for every compactly supported smooth function $\eta:\widetilde{M}\to \RR$
then the isoperimetric  inequality 
\begin{equation}\tag{$I_{\rho}$}\label{equation : discrete weighted Poincare inequality}
\sum_{x\in A} \rho_d(x)\le C\#\partial A
\end{equation}
holds for 
every finite subset $A\subset \Gamma$ where $\rho_d(x)=\rho(xF)$.
\end{lemma}
\begin{proof}
Assume the contrary. Then there is a sequence of sets $A_i\subset \Gamma$ such that
$$\frac{\#\partial A_i}{\sum_{x\in A_i}\rho_d(x)}\ \ \longrightarrow\ \  0.$$
As in  \cite{brooks} we can construct smooth functions
on the universal cover, which exhibit the same asymptotic behavior.
Choose $0<\varepsilon_0<\varepsilon_1<1$, where both $\varepsilon_i$ are sufficiently small,
and a function $\kappa:[0,1]\to [0,1]$ such that $\kappa(t)=1$ when $t\ge \varepsilon_1$ and $\kappa(t)=0$ when
$t\le \varepsilon_0$. 

Let $\chi_i$ be the characteristic function of $B_i=\bigcup_{\gamma\in A_i}\gamma F$
and let $\eta_i=\chi_i\,\kappa(d(x,\widetilde{M}\setminus B_i))$. Then one can estimate 
$$\int \vert \nabla\eta_i\vert^2\le C \mathrm{Vol}(\partial B_i)\le C' \#\partial A_i.$$
On the other hand 
$$\int  \eta_i^2(x) \rho(x)\ge D\sum_{\gamma\in A_i}\int_{\gamma F} \rho(x)\ge D \Vol(F)
\sum_{\gamma\in A_i}\rho_d(\gamma).$$
Consequently 
$$\frac{\int \vert \nabla\eta\vert^2}{\int \eta(x)^2\rho(x)}\ \ \longrightarrow\ \  0$$
and we reach a contradiction.
\end{proof}

We also need an unpublished theorem of Block and Weinberger (see \cite{whyte} for 
a proof).
\begin{theorem}[Block-Weinberger, Whyte]
Let $c=\sum_{x\in X}c_x[x]\in C_0^{\uf}(X)$. Then $[c]=0$ in $H_0^{\uf}(X)$ if and only
if the inequality $$\left\vert\ \sum_{x\in A}c_x\ \right\vert \le C\#\partial A$$ holds for every finite $A\subset \Gamma$
and some $C>0$.
\end{theorem}
We observe that by inequality 
(\ref{equation : discrete weighted Poincare inequality})
and the above theorem the class represented by 
$c=\sum_{x\in X}\rho_d(x)[x] \in C_0^{\uf}(X)$ 
vanishes in the uniformly  finite homology group.  We will denote 
by $[1/f]$ the homology class represented by the chain $\sum_{x\in X}1/f(\vert x\vert)$.
Theorem 
\ref{theorem : weighted Poincare implies vanishing} now follows from the following

\begin{lemma}
Let $f$ be non-decreasing slowly growing function.
If $\left[1/f\right]=0$ in $H_0^{\uf}(X)$ then $[X]=0$ in $H_0^f(X)$.
\end{lemma}

\begin{proof}
Let $1/f(\vert x\vert)=\partial \phi(x)$, where $\phi\in C_0^{uf}(X)$ and $\mathscr{P}(\phi)=1$.
Let $\psi(x,y)=\phi(x,y)f(\vert x,y\vert)$. Then $\psi\in C_1^f(X)$ and
\begin{eqnarray*}
\partial \psi(x)&=&\sum_{y\in B(x,1)}\phi(y,x)f(\vert(x,y)\vert)\\\\
&=&\sum_{y\,(+)}\phi(y,x)f(\vert(x,y)\vert)-\sum_{y\,(-)}\phi(x,y)f(\vert(x,y)\vert)
\end{eqnarray*}
where $(+)$ and $(-)$  denote sums over edges $(y,x)$ with positive coefficients and negative 
coefficients respectively.
Continuing we have
\begin{eqnarray*}
&\ge&\sum_{y\,(+)} \phi(y,x)f(\vert x\vert)-\sum_{y\,(-)}\phi(x,y)f(\vert x\vert+1)
\end{eqnarray*}
Let $\mathcal{N}$ be such that $\#B(x,1)\le \mathcal{N}$
for every $x\in X$.  Fix $0<C<1$ and let $\varepsilon>0$ be such that 
$1-\varepsilon\Vert\phi\Vert_{\infty}\, \mathcal{N}>C$.
By assumption there exists a compact set 
$K\subset X$ such that outside of $K$
we have $f(\vert x\vert+1)\le f(\vert x\vert)+\varepsilon$.
By factoring $f(\vert x\vert)$ from first two sums we have
\begin{eqnarray*}
&=&f(\vert x\vert)\left(\sum_{y\in B(x,1)} \phi(y,x)\right)-\varepsilon\left(\sum_{y\,(-)}\phi(x,y)\right)\\\\
&\ge&\partial\phi(x)f(\vert x\vert)-\varepsilon\left(\sum_{y\,(-)}\phi(x,y)\right)\\\\
&\ge&1-\varepsilon\Vert\phi\Vert_{\infty}\, \mathcal{N}\\\\
&\ge&C
\end{eqnarray*}
for every $x\in X\setminus K$. 
To achieve the same lower bound on $K$ we can add finitely many tails $t_x$ to $\psi$ to 
assure $\partial\psi(x)\ge C$ on $K$. We do not alter any of the previously
ensured properties since $K$ is finite. The 1-chain $\psi'$ that we obtain in this way 
satisfies
$\psi'\in C_1^f(X)$ and $\partial\psi'(x)\ge C$, thus by Lemma \ref{lemma : bounded below chains vanish then [G] vanishes} the fundamental class vanishes as well. 
\end{proof}

Theorem \ref{theorem : weighted Poincare implies vanishing} can be generalized to
open manifolds under suitable assumptions. We leave the details to the reader.

\begin{remark}\normalfont
We would like to point out that it is not hard to show that the isoperimetric inequality
(\ref{equation : Zuk-type isoperimetric inequality}) is implied by a discrete inequality 
$$\sum_{x\in X}\vert \eta(x)\vert^2\le C\sum_{d(x,y)\le 1}\vert \eta(x)-\eta(y)\vert^2f(\vert(x,y)\vert)$$
for all finitely supported $\eta:X\to\RR$. The latter can be interpreted as the existence of a 
spectral gap for a weighted discrete Laplace operator.
\end{remark}

\section{Primitives of differential forms}
Sullivan in \cite{sullivan} studied the growth of a primitive of a differential form 
and asked a question about the connection of a certain isoperimetric inequality and
the existence of a bounded primitive of the volume form on a non-compact manifold. 
This question was later answered by Gromov \cite{gromov-proceedings} and other proofs are
provided by Brooks \cite{brooks} and Block and Weinberger \cite{block-weinberger}.
For unbounded primitives this question was studied by Sikorav \cite{sikorav}, and \.{Z}uk \cite{zuk}.
 
As an application of the controlled coarse homology we obtain precise estimates on 
growth of primitives of the volume form on 
covers of compact Riemannian manifolds.

\begin{theorem}
Let $N$ be an open, complete Riemannian manifold of bounded geometry and let $X\subset N$ be 
a discrete subset of $N$ which is quasi-isometric to $N$. Then 
$[X]=0$ in $H_0^{f}(X)$ if and only if the volume form on $N$ has a primitive
of growth controlled by $f$.
\end{theorem}

\begin{proof}
First note that by quasi-isometry invariance of the controlled homology (corollary \ref{corollary : QI invariance}), it is sufficient to prove the statement for any $X$ that is quasi-isometric to $N$.

We follow Whyte's proof \cite[Lemma 2.2.]{whyte-ahat}. Let $\kappa>0$ be smaller
than the convexity radius of $N$. We choose a maximal $\kappa$-separated subset
 $X\subset N$
 and consider the partition of unity $\set{\varphi_x}_{x\in X}$ associated to the cover by balls centered 
 at points of $X$ of radius $\kappa$.  Denote $v_x=\varphi_x v$.

Given $\psi\in C_1^f(X)$  of propagation $\mathscr{P}(\psi)\le r$, which bounds the fundamental 
class we note that $v_x-v_y=d\omega_{(x,y)}$, where $\omega_{(x,y)}$ are 
 $(n-1)$-forms of uniformly bounded supports. If we let $\omega=\sum_{d(x,y)\le r}\psi(x,y)\omega_{(x,y)}$
 then we have

$$d\omega=\sum_{d(x,y)\le r}\psi(x,y)(v_x-v_y)=\sum_{x\in X} \widetilde{\partial}\psi_x v_x=2v.$$

On the other hand, if $v=d\omega$ and $\vert\omega\vert\le Cf$ then it follows from Stokes' theorem
that the isoperimetric inequality with $f$ holds for $X$. By Theorem 
\ref{theorem : vanishing iff inequality holds} this implies vanishing of the fundamental class
in $H_0^f(X)$.
\end{proof}

The above gives a different proof of a theorem of Sikorav \cite{sikorav}  and also explains the nature of
the extra constants appearing in the formulation of the main theorem of that paper. These constants
reflect the fact that the growth of the primitive of a differential form is of large-scale
geometric nature.

\smallskip

For the rest of this section, we fix the notation as follows:
$M$ is a compact manifold with a universal cover $\widetilde{M}$
and fundamental group $\pi_1(M)$.

\begin{corollary}
Let $M$ be a compact manifold with $\Gamma=\pi_1(M)$.
Then $[\Gamma]=0$ in $H_0^f(\Gamma)$ if and only if the volume form on the universal cover
$\widetilde{M}$ has a primitive whose growth is controlled by $f$.
\end{corollary}

For various classes of groups we obtain specific estimates based on the results from
previous sections. For instance for polycyclic groups we have the 
following

\begin{corollary}
Let $\pi_1(M)$ be an infinite, polycyclic group. Then 
the primitive of the volume form on $\widetilde{M}$ has exactly linear growth.
\end{corollary}

For groups with finite asymptotic dimension the dichotomy between amenable and non-amenable 
groups is manifested in a gap between possible growth type of the primitives of volume forms. 
We discuss only the Baumslag-Solitar groups 
$$\mathrm{BS}(m,n)=\langle a,b\ \vert\ ab^ma^{-1}=b^n\rangle,$$
which have finite asymptotic
dimension of linear type and are either solvable or have a non-abelian 
free subgroup.

\begin{corollary}
Let $M$ be a compact manifold with $\pi_1(M)=\mathrm{BS}(m,n)$. 
Then a primitive of the volume form on $\widetilde{M}$ has
\begin{enumerate}
\item bounded growth if $\pi_1(M)$ is non-amenable i.e. $\vert m\vert\neq 1\neq \vert n\vert$,
\item  linear growth if $\pi_1(M)$ is amenable i.e., $\vert m\vert=1$ or $\vert n\vert=1$.
\end{enumerate}
\end{corollary}

\section{Other applications and final remarks}

\subsection{Distortion of subgroups} Let $G\subseteq \Gamma$ be a subgroup. Then the inclusion is always a coarse embedding and 
it is often a question whether 
$G$ is undistorted in $\Gamma$, that is the inclusion is a quasi-isometric embedding.
We have the following
\begin{proposition}
Let $G\subseteq \Gamma$ be a inclusion of  a subgroup with both $G$ and $\Gamma$ finitely
generated. If $[G]=0$ in $H_0^{f}(G)$ then $[\Gamma]=0$ in $H_0^{f\circ \rho_-^{-1}}(\Gamma)$.
\end{proposition}
\begin{proof}
It is straightforward to check that the pushforward of a $1$-chain that bounds the fundamental class $[G]$ in $H_0^f(G)$ will have growth controlled by $f\circ \rho_-^{-1}$ on $\Gamma$. We then partition $\Gamma$ into $G$-cosets, and construct a $1$-chain that bounds $[\Gamma]$ as the sum of translates of the chain that bounds $[G]$.
\end{proof}
In particular if $G$ is undistorted in $\Gamma$, then the fundamental classes should vanish in 
appropriate groups with the same control function $f$. This fact can be used to estimate the 
distortion of a subgroup or at least to decide whether a given subgroup can be an undistorted 
subgroup of another group, however except the iterated wreath
products mentioned earlier we do not know  examples in which explicit computation would 
be possible.

\subsection{Questions}
A natural question is whether one can introduce a homology or cohomology
theory which would in a similar way reflect isoperimetric properties of amenable actions. Isoperimetric
inequalities for such actions are discussed in \cite{nowak-isoactions}.\\

\end{document}